\documentclass[a4paper,10pt]{article}

\usepackage[utf8]{inputenc}
\usepackage{amsmath}
\usepackage{amssymb}
\usepackage{MnSymbol}
\usepackage{graphicx}
\usepackage{subcaption}
\usepackage{fancyhdr}
\usepackage[shortlabels]{enumitem}
\usepackage[nottoc,numbib]{tocbibind}
\usepackage{xcolor}
\usepackage{aliascnt}
\usepackage{hyperref}
\usepackage{cleveref}
\usepackage{mathtools}
\usepackage{amsfonts}
\usepackage{etoolbox}
\usepackage{amsthm}
\usepackage{pifont}

\newcommand{\setseparator}{\,:\,}
\newcommand{\converges}[2][\infty]{\underset{{#2}\rightarrow{#1}}{\longrightarrow}}
\newcommand{\differential}[1]{\,\mathrm{d}{#1}}
\newcommand{\inner}[3][\hilbertspace]{\left<{#2},{#3}\right>_{#1}}
\newcommand{\dualmap}[3]{\left({#2},{#3}\right)_{#1}}
\DeclareMathOperator{\vecspan}{span}

\DeclareMathOperator{\supp}{supp}
\DeclarePairedDelimiter\norm{\lVert}{\rVert}
\DeclarePairedDelimiter\modulus{\lvert}{\rvert}
\newcommand{\N}{\mathbb{N}}
\newcommand{\R}{\mathbb{R}}

\newcommand{\leb}[1]{l^{#1}}
\newcommand{\hilbertspace}{\mathcal{H}}
\newcommand{\banachspace}{\mathcal{B}}
\newcommand{\error}{\mathcal{E}}

\newcommand{\thetitle}{When is there a Representer Theorem?\\Reflexive Banach spaces}
\newcommand{\theauthor}{Kevin Schlegel}
\newcommand{\theaddress}{Mathematical Institute\\
University of Oxford\\
Andrew Wiles Building, Radcliffe Observatory Quarter\\
Woodstock Road, Oxford, OX2 6GG, UK}
\newcommand{\theemail}{schlegel@maths.ox.ac.uk}

\makeatletter
\newtheoremstyle{theorem}
  {5mm}
  {7mm}
  {\addtolength{\@totalleftmargin}{7mm}
   \addtolength{\linewidth}{-7mm}
   \parshape 1 7mm \linewidth}
  {-7mm}
  {\bfseries}
  {}
  {\newline}
  {\normalfont\textbf{\thmname{#1}\thmnumber{ #2}}\textit{\thmnote{ (#3)}}}

\makeatother
\theoremstyle{theorem}

\newtheorem{theorem}{Theorem}[section]
\AfterEndEnvironment{theorem}{\noindent\ignorespaces}
\crefname{theorem}{theorem}{theorems}
\newaliascnt{proposition}{theorem}
\newtheorem{proposition}[proposition]{Proposition}
\aliascntresetthe{proposition}
\AfterEndEnvironment{proposition}{\noindent\ignorespaces}
\crefname{proposition}{proposition}{propositionss}
\newaliascnt{lemma}{theorem}
\newtheorem{lemma}[lemma]{Lemma}
\aliascntresetthe{lemma}
\AfterEndEnvironment{lemma}{\noindent\ignorespaces}
\crefname{lemma}{lemma}{lemmas}
\newaliascnt{remark}{theorem}

\aliascntresetthe{remark}
\AfterEndEnvironment{remark}{\noindent\ignorespaces}
\crefname{remark}{remark}{remarks}
\newaliascnt{corollary}{theorem}

\aliascntresetthe{corollary}
\AfterEndEnvironment{corollary}{\noindent\ignorespaces}
\crefname{corollary}{corollary}{corollaries}
\newaliascnt{definition}{theorem}
\newtheorem{definition}[definition]{Definition}
\aliascntresetthe{definition}
\AfterEndEnvironment{definition}{\noindent\ignorespaces}
\crefname{definition}{definition}{definitions}
\newaliascnt{example}{theorem}

\aliascntresetthe{example}
\AfterEndEnvironment{example}{\noindent\ignorespaces}
\crefname{example}{example}{examples}

\newcommand{\proofend}{\hspace*{\fill}\ding{113}\\}

\let\oldendproof\endproof
\renewcommand{\endproof}{\proofend\oldendproof}

\makeatletter
\newtheoremstyle{notation}
  {5mm}
  {5mm}
  {\addtolength{\@totalleftmargin}{7mm}
   \addtolength{\linewidth}{-7mm}
   \parshape 1 7mm \linewidth}
  {-7mm}
  {\bfseries}
  {}
  {\newline}
  {\normalfont\textbf{\thmname{#1}\thmnumber{ #2}:}\thmnote{ (#3)}}

  \makeatother
\theoremstyle{notation}

\crefname{notation}{notation}{notations}

\newtheoremstyle{subproof}
  {0mm}
  {5mm}
  {}
  {}
  {\bfseries}
  {}
  {\newline}
  {\normalfont\textit{\textbf{\thmname{#1}\thmnumber{ #2}:}\thmnote{ (#3)}}}

\theoremstyle{subproof}

\newtheorem{subproof}{Part}
\AfterEndEnvironment{subproof}{\noindent\ignorespaces}
\crefname{subproof}{part}{parts}

\hypersetup{
  colorlinks=false,
  linkbordercolor=red,
  pdfborderstyle={/S/U/W .5}
}

\fancypagestyle{paper}{
  \fancyhead{}
  \fancyhead[L]{\footnotesize\nouppercase{\leftmark}}
  \fancyhead[R]{\footnotesize\theauthor}
  \fancyfoot{}
  \fancyfoot[C]{\thepage}
}

\renewcommand{\maketitle}{
  \thispagestyle{empty}
  \begin{center}
    {\Large \thetitle \par}
    \vspace{5mm}
    {\large \theauthor \par}
    \vspace{2mm}
    {\small \theaddress \par}
    {\small Email: \textit{\theemail} \par}
    \vspace{3mm}
    \today
    \vspace{8mm}
  \end{center}
}
\begin{document}

\maketitle

\pagestyle{paper}
{\Large\bf Abstract}\\
We consider a general regularised interpolation problem for learning a parameter vector from data. The well known representer theorem says that under certain conditions on the regulariser there exists a solution in the linear span of the data points. This is at the core of kernel methods in machine learning as it makes the problem computationally tractable. Most literature deals only with sufficient conditions for representer theorems in Hilbert spaces. We prove necessary and sufficient conditions for the existence of representer theorems in reflexive Banach spaces and illustrate why in a sense reflexivity is the minimal requirement on the function space. We further show that if the learning relies on the linear representer theorem, then the solution is independent of the regulariser and in fact determined by the function space alone. This in particular shows the value of generalising Hilbert space learning theory to Banach spaces.

{\bf Keywords:} representer theorem, regularised interpolation, regularisation, kernel methods, reproducing kernel Banach spaces
\section{Introduction}
It is a common approach in learning theory to formulate a problem of estimating functions from input and output data as an optimisation problem. Most commonly used is regularisation, in particular \textit{Tikhonov regularisation}  where we consider an optimisation problem of the form
$$\min\left\{\mathcal{E}({(\inner{f}{x_i},y_i)}^m_{i=1}) + \lambda\Omega(f)\setseparator f\in\hilbertspace\right\}$$
where $\hilbertspace$ is a Hilbert space with inner product $\inner{\cdot}{\cdot}$, $\left\{(x_i,y_i) \setseparator i\in\N_m\right\}\subset\hilbertspace\times Y$ is a set of given input/output data with $Y\subseteq\R$, $\mathcal{E}\colon\R^m\times Y^m\rightarrow\R$ is an \textit{error function}, $\Omega\,\colon\hilbertspace\rightarrow\R$ a \textit{regulariser} and $\lambda>0$ is a \textit{regularisation parameter}.\\
The use of a regulariser $\Omega$ is often described as adding additional information or using previous knowledge about the solution to solve an ill-posed problem or to prevent an algorithm from overfitting to the given data. This makes it an important method for learning a function from empirical data from a very large class of functions. Problems of this kind appear widely, in particular in supervised and semisupervised learning, but also in various other disciplines wherever empirical data is produced and has to be explained by a function. This has motivated the study of regularisation problems in mathematics, statistics and computer science, in particular machine learning (Cucker and Smale~\cite{smale2001}, Shawe-Taylor and Cristianini~\cite{shawe-taylor2004}, Micchelli and Pontil~\cite{micchelli2005}).\\
\ \\
It is commonly stated that the regulariser favours certain desirable properties of the solution and can thus intuitively be thought of as picking the function that may explain the data and which is the simplest in some suitable sense. This is in analogy with how a human would pick a function when seeing a plot of the data. One contribution of this work is to clarify this view as we show that if the learning relies on the linear representer theorem the solution is in fact \textit{independent of the regulariser} and it is \textit{the function space} we chose to work in which \textit{determines the solution}.\\
\ \\
Regularisation has been studied in particular in Hilbert spaces as stated above. This has various reasons. First of all the existence of inner products allows for the design of algorithms with very clear geometric intuitions often based on orthogonal projections or the fact that the inner product can be seen as a kind of similarity measure.\\
But in fact crucial for the success of regularisation methods in Hilbert spaces is the well known \textit{representer theorem} which states that for certain regularisers there is always a solution in the linear span of the data points (Kimeldorf and Wahba~\cite{kimeldorf1971}, Cox and O'Sullivan~\cite{cox1990}, Sch\"{o}lkopf and Smola~\cite{smola1998,smola2001}). This means that the problem reduces to finding a function in a finite dimensional subspace of the original function space which is often infinite dimensional. It is this dimension reduction that makes the problem computationally tractable.\\
Another reason for Hilbert space regularisation finding a variety of applications is the \textit{kernel trick} which allows for any algorithm which is formulated in terms of inner products to be modified to yield a new algorithm based on a different symmetric, positive semidefinite kernel leading to learning in \textit{reproducing kernel Hilbert spaces (RKHS)} (Sch\"{o}lkopf and Smola~\cite{smola2002}, Shawe-Taylor and Cristianini~\cite{shawe-taylor2004}). This way nonlinearities can be introduced in the otherwise linear setup. Furthermore kernels can be defined on input sets which a priori do not have a mathematical structure by embedding the set into a Hilbert space.\\
\ \\
The importance and wide success of kernel methods and the representer theorem have led to extensions of the theory to \textit{reproducing kernel Banach spaces (RKBS)} by Zhang, Xu and Zhang~\cite{zhang2009} and representer theorems for learning in RKBS by Zhang and Zhang~\cite{zhang2012}. This motivates the study of the more general regularisation problem
\begin{equation}
  \label{eq:regularisation_problem}
  \min\left\{\mathcal{E}({(L_i(f),y_i)}^m_{i=1}) + \lambda\Omega(f)\setseparator f\in\banachspace\right\}
\end{equation}
where $\banachspace$ is a reflexive Banach space and the $L_i$ are continuous linear functionals on $\banachspace$. We are considering reflexive Banach spaces for two reasons. Firstly they are the fundamental building block of reproducing kernel Banach spaces as can be seen in \cref{sec:examples} where we state the most relevant definitions and results from the work of Zhang, Xu and Zhang~\cite{zhang2009}. Secondly we show in \cref{sec:consequences} that for the setting considered reflexivity is the minimal assumption on the space $\banachspace$ for which our results can hold.\\
\pagebreak
\ \\
Classical statements of the representer theorem give sufficient conditions on the regulariser for the existence of a solution in the linear span of the representers of the data. Argyriou, Micchelli and Pontil~\cite{argyriou2009} gave the, to our knowledge, first attempt in proving necessary conditions to classify all regularisers which admit a linear representer theorem. They prove a necessary and sufficient condition for differentiable regularisers on Hilbert spaces. In the authors earlier work~\cite{schlegel2019} this result was extended part way to not necessarily differentiable regularisers on uniformly convex and uniformly smooth Banach spaces.\\
In this paper we answer the question of existence of representer theorems in a sense completely by extending those results to reflexive Banach spaces and showing optimality of reflexivity. An important consequence of our characterisation of regularisers which admit a linear representer theorem is that one can now prove that in fact the solution does not depend on the regulariser but only on the space the optimisation problem is stated in. This is interesting for two reasons. Firstly it means that we can always pick the regulariser best suited for the application at hand, whether this is computational efficiency or ease of formal calculations. Secondly it further illustrates the importance of being able to learn in a larger variety of spaces, i.e. of extending the learning theory to a variety of Banach spaces.\\
\ \\
In \cref{sec:preliminaries} we will introduce the relevant notation and mathematical background needed for our main results. In particular we will present the relevant results of Argyriou, Micchelli and Pontil~\cite{argyriou2009} which justify focusing on the easier to study regularised interpolation problem rather than the general regularisation problem.\\
Subsequently in \cref{sec:representer_theorem} we will present one of the main results of our work that regularisers which admit a linear representer theorem are \textit{almost radially symmetric} in a way that will be made precise in the statement. We state and prove two lemmas which capture most of the important structure required and then give the proof of the theorem.\\
In \cref{sec:consequences} we discuss the consequences of the theorem from \cref{sec:representer_theorem}. We prove the other main result of the paper, which states that, if we rely on the linear representer theorem for learning, in most cases the \textit{solution is independent of the regulariser} and \textit{depends only on the function space}. We also illustrate why it is clear that we cannot hope to weaken the assumption on the space $\banachspace$ any further than reflexivity.\\
Finally in \cref{sec:examples} we give some examples of spaces to which our results apply. This section is based on the work of Zhang, Xu and Zhang~\cite{zhang2009} and Zhang and Zhang~\cite{zhang2012} on reproducing kernel Banach spaces so we will first be presenting the relevant definitions and results on the construction of RKBS from~\cite{zhang2009} in this section. We then give a few examples which have been presented in these papers.

\section{Preliminaries}\label{sec:preliminaries}
In this section we present the notation and theory used to state and prove our main results. We summarise results which allow us to reduce the problem to study regularised interpolation problems and present the theory of duality mappings required for the proofs of our main results.\\
\ \\
Throughout the paper we use $\N_m$ to denote the set $\{1,\ldots,m\}\subset\N$ and $\R^+$ to denote the non-negative real line $[0,\infty)$.\\
We will assume we have $m$ data points $\left\{(x_i,y_i) \setseparator i\in\N_m\right\}\subset\banachspace\times Y$, where $\banachspace$ will always denote a reflexive real Banach space and $Y\subseteq\mathbb{R}$. Typical examples of $Y$ are finite sets of integers for classification problems, e.g. $\{-1,1\}$ for binary classification, or the whole of $\R$ for regression.

\subsection{Regularised Interpolation}\label{sec:interpolation}
As discussed in the introduction we are interested in problems of the form~(\ref{eq:regularisation_problem}), namely
\begin{equation*}
  \min\left\{\error({(L_i(f),y_i)}^m_{i=1}) + \lambda\Omega(f)\setseparator f\in\banachspace\right\}
\end{equation*}

where $\banachspace$ is a reflexive Banach space. The $L_i$ are continuous linear functionals on $\banachspace$ with the $y_i\in Y\subseteq\R$ the corresponding output data. The functional $\error\,\colon\R^m\times Y^m\rightarrow\R$ is an \textit{error functional}, $\Omega\,\colon\banachspace\rightarrow\R$ a \textit{regulariser} and $\lambda>0$ is a \textit{regularisation parameter}.\\
Argyriou, Micchelli and Pontil~\cite{argyriou2009} show in the Hilbert space case that under very mild conditions this regularisation problem admits a linear representer theorem if and only if the regularised interpolation problem
\begin{equation}
\label{eq:interpolation_problem}
  \min\left\{\Omega(f)\setseparator f\in\banachspace, L_i(f)=y_i\,\forall i\in\N_m\right\}
\end{equation}

admits a linear representer theorem. This is not surprising as the regularisation problem is more general and one obtains a regularised interpolation problem in the limit as the regularisation parameter goes to zero.\\
More precisely they proved the following theorem for the Hilbert space setting. The proof of this theorem for the generality of the setting of this paper is almost identical to the version given in Argyriou, Micchelli and Pontil~\cite{argyriou2009} but requires a few adjustments. The full proof is presented in the appendix.

\begin{theorem}\label{thm:problem_equivalence}
Let $\error$ be a lower semicontinuous error functional which is bounded from below. Assume further that for some $\nu\in\R^m\setminus\{0\}, y\in Y^m$ there exists a unique minimiser $0\neq a_0\in\R$ of $\min\{\error(a\nu,y)\setseparator a\in\R\}$.\\
Assume the regulariser $\Omega$ is  lower semicontinuous and has bounded sublevel sets.\\
Then $\Omega$ is admissible for the regularised interpolation problem (\ref{eq:interpolation_problem}) if the pair $(\error,\Omega)$ is admissible for the regularisation problem (\ref{eq:regularisation_problem}).
\end{theorem}
Note that the assumptions on the error function and regulariser presented here are as in the paper~\cite{argyriou2009}. It is remarked in that paper that other conditions can also be sufficient. The proof of this result follows the earlier mentioned concept that one obtains a regularised interpolation problem as the limit of regularisation problems.\\
It is worth noting that the reverse direction of above theorem does not require any assumptions on the error function or regulariser. In fact we have the following result.
\begin{proposition}
  Let $\error, \Omega$ be an arbitrary error functional and regulariser satisfying the general assumption that minimisers always exist. Then the pair $(\error,\Omega)$ is admissible for the regularisation problem (\ref{eq:regularisation_problem}) if $\Omega$ is admissible for the regularised interpolation problem (\ref{eq:interpolation_problem}).
\end{proposition}
This allows us to focus on the regularised interpolation problem which is a lot easier to study and yet obtain information about the regularisation problem which is more relevant for application. In particular every representer theorem proved below for regularised interpolation is valid for regularisation problems without any restrictions.

\subsection{Duality mappings}
The results in the authors work~\cite{schlegel2019} made clear that the representer theorem is essentially a result about the dual space and the proofs heavily rely on tangents to balls which are exactly described by the duality mapping. Hence to generalise the results from~\cite{schlegel2019} further we need some definitions and results about duality mappings which are given in this section.

\begin{definition}[Duality mapping]
  Let $\mu:\R^+\rightarrow\R^+$ a continuous and strictly increasing function such that $\mu(0)=0$ and $\mu(t)\converges{t}\infty$.\\
  A set-valued map $J:V \rightarrow 2^{V^\ast}$ is called a duality mapping of $V$ into $V^\ast$ with gauge function $\mu$ if $J(0)=\{0\}$ and for $0\neq x\in V$
  $$J(x) = \left\{L\in V^\ast \setseparator L(x) = \norm{L}\cdot\norm{x}, \norm{L}=\mu(\norm{x}) \right\}$$
\end{definition}
In this work we will be considering the case where $\mu$ is the identity and the duality mapping an isometry.\\
The following properties of the duality mapping are well known and can be found e.g.\ in Dragomir~\cite{dragomir2004} and the references therein.

\begin{proposition}\label{prop:dual_map_basics}
  For every $x\in V$ the set $J(x)$ is nonempty, closed and convex. Furthermore we have the following equivalences.
  \begin{enumerate}
  \item $J$ is surjective if and only if $V$ is reflexive.
  \item $J$ is injective if and only if $V$ is strictly convex.
  \item $J$ is univocal if and only if $V$ is smooth.
  \item $J$ is norm-to-weak* continuous exactly at points of smoothness of $V$.
  \end{enumerate}
\end{proposition}

The following generalised version of the Beurling-Livingston theorem is essential for the proof of our main result. A proof of this theorem can be found in the work by Browder~\cite{browder1965} which is very general, deducing the result from a result on multi-valued monotone nonlinear mappings. A more direct proof, giving the interested reader a better idea of the objects occurring in the result, can be found in the work by Bla\u{z}ek~\cite{blazek1982}. Unfortunately there is an issue in the proof in the paper by Bla\u{z}ek, we present a corrected version of it in the appendix of this paper. The overall intuition of Bla\u{z}eks proof is correct nonetheless and a moral summary of it can also be found in a paper by Asplund~\cite{asplund1967}.

\begin{theorem}[Beurling-Livingston]
\label{thm:duality_map_intersection}
Let $V$ be a real normed linear space with duality mapping $J$ with gauge function $\mu$ and $W$ a reflexive subspace of $V$.\\
Then for any fixed $x_0\in V, L_0\in V^\ast$ there exists $z\in W$ such that
  $$J(x_0 + z) \cap (W^\perp - L_0) \neq \emptyset$$
  where $W^\perp$ denotes the annihilator of $W$ in $V^\ast$.
\end{theorem}
\vspace*{-10mm}
\section{Existence of Representer Theorems}\label{sec:representer_theorem}
We are now in the position to present the first of the main results of this paper. Throughout this section $\banachspace$ will denote a reflexive Banach space with dual space $\banachspace^\ast$ and duality mapping
$$J(x) = \left\{L\in\banachspace^\ast \setseparator L(x)=\norm{L}\cdot\norm{x}, \norm{L}=\norm{x}\right\}$$
It is well known that this mapping is surjective for reflexive Banach spaces but need not be injective or univocal. It is injective if and only if the space is strictly convex and univocal if and only if the space is smooth.\\
As argued in \cref{sec:interpolation}, rather than studying the general regularisation problem~(\ref{eq:regularisation_problem}) we can consider the regularised interpolation problem~(\ref{eq:interpolation_problem}), namely
\begin{equation*}
  \min\left\{\Omega(f) \setseparator f\in\banachspace,\, L_i(f)=y_i\,\forall i\in\N_m\right\}
\end{equation*}
Note that the $L_i$ could be point evaluations $L_i(f)=f(x_i)$, in which case the problem reduces to usual function evaluations at data points $x_i$, but the framework also allows for other linear functionals such as e.g.\ local averages of the form $L(f) = \int_{\banachspace}f(x)\differential{P(x)}$ where $P$ is a probability measure on $\banachspace$.\\
Our goal is to classify all regularisers for which there exists a linear representer theorem. As stated in the authors earlier work~\cite{schlegel2019}, both our work as well as previous work by Micchelli and Pontil~\cite{miccelli2004} indicate that the representer theorem in its core is actually a result about the dual space. This does not become apparent in its classical form as in a Hilbert space the dual element is the element itself. Thus in the Banach space setting considered in this work we formulate the representer theorem in terms of dual elements of the data, as done in~\cite{schlegel2019}. In this paper the space might not be smooth so we also need to account for the dual map potentially not being univocal. We call regularisers which always allow a solution which has a dual element in the linear span of the linear functionals defining the problem admissible, as is made precise in the following definition.
\begin{definition}[Admissible Regularizer]\label{def:admissibility}
  We say a function $\Omega:\banachspace\rightarrow\R$ is admissible if for any $m\in\N$ and any given data $\{L_1,\ldots,L_m\}\subset\banachspace^\ast$ and $\{y_1,\ldots,y_m\}\subset Y$ such that the interpolation constraints can be satisfied the regularised interpolation problem \cref{eq:interpolation_problem} admits a solution $f$ such that there exist coefficients $\{c_1,\ldots,c_m\}\subset\R$ such that
  $$\hat{L}=\sum\limits_{i=1}^m c_i L_i \in J(f)$$
\end{definition}

It is well known that being a non-decreasing function of the norm on a Hilbert space is a sufficient condition for the regulariser to be admissible. By a Hahn-Banach argument similar as e.g.\ in Zhang, Zhang~\cite{zhang2012} this generalises to this notion of admissibility for reflexive Banach spaces.\\
We want to show that an admissible regulariser is in a sense almost radially symmetric, similar to our previous work~\cite{schlegel2019}. The proof strategy is similar to the one in~\cite{schlegel2019} but it will turn out that in particular a lack of strict convexity makes the situation a lot more delicate to deal with. We will begin by showing that admissible regularisers are still nondecreasing along tangents in Banach spaces which are not strictly convex, but in a weaker form than for uniform Banach spaces. Subsequently we will explore to what extend this weaker tangential bound still implies radial symmetry.

\begin{lemma}
  \label{lma:tangential_bound}
  A function $\Omega\,\colon\banachspace\rightarrow\R$ is admissible if and only if for every exposed face of the ball $\Omega$ attains its minimum in at least one point, and for every $f$ in the face where the minimum is attained and every $L\in J(f)$ exposing the face and every $f_T\in\ker(L)$ we have
  $$\Omega(f+f_T) \geq \Omega(f)$$
\end{lemma}
\vspace*{-10mm}
\begin{definition}
  We are going to refer to the points that \cref{lma:tangential_bound} applies to as \textit{admissible points}.
\end{definition}
Note that this in particular means that every exposed point is admissible and the bound applies to every functional exposing it. Further, if the point is rotund then the lemma applies to every functional attaining its norm at the point.\\

\begin{proof}[Of~\cref{lma:tangential_bound}]
  \begin{subproof}[$\Omega$ admissible $\Rightarrow$ nondecreasing along tangential directions]
    \label{spf:tangential_bound_part1}
    Fix any $f\in\banachspace$ and consider, for $L\in J(f)$ arbitrary but fixed, the regularised interpolation problem
    $$\min\left\{\Omega(g) \setseparator g\in\banachspace, L(g)=L(f)=\norm{f}^2\right\}$$
    Since $\Omega$ is admissible there exists a solution $f_0$ such that $c\cdot L\in J(f_0)$. Now if there does not exist $g\in\banachspace$ such that $g\neq f$ and $L\in J(g)$ then this can only be $f$ itself, as in the case of uniform Banach spaces~\cite{schlegel2019}. Thus for any $f_T\in\ker(L)$ also $L(f+f_T)=L(f)=\norm{f}^2$ and $f+f_T$ also satisfies the constraints and hence necessarily $\Omega(f+f_T) \geq \Omega(f)$.\\
    But if there exists $g\in\banachspace$ such that $L\in J(g)$ we have no way of making a statement about how $\Omega(f)$ and $\Omega(g)$ compare. All we can say is that in this face containing $f$ and $g$ there is at least one point where the minimum of $\Omega$ is attained. It is clear that for any of those minimal points the above discussion is true for $L$ exposing the face so that we obtain the claimed tangential bound.
  \end{subproof}

  \begin{subproof}[Nondecreasing along tangential directions $\Rightarrow$ $\Omega$ admissible]
    \label{spf:tangential_to_admissible_reflexive}
    Conversely fix any data $(L_i,y_i)\in\banachspace^\ast\times Y$ for $i\in\N_m$ such that the constraints can be satisfied. Let $f_0$ be a solution to the regularised interpolation problem. If $\vecspan\{L_i\}\cap J(f_0) \neq \emptyset$ we are done, so assume not. We let
  $$Z = \left\{f_T\in\banachspace \setseparator L_i(f_T)=0\,\forall i\in\N_m\right\} = \bigcap\limits_{i\in\N_m}\ker{L_i}$$
  We want to show that there exists $f_T\in Z$ such that $\vecspan\{L_i\}\cap J(f_0+f_T) \neq \emptyset$.\\
  To see that this is true choose $V=\banachspace$ and $W=Z$ in the Beurling-Livingston theorem~(\cref{thm:duality_map_intersection}). Since $Z$ is a closed subspace of a reflexive space it is itself reflexive. Further choose $x_0 = f_0$ and $L_0=0$. Then the theorem says that there exists $f_T\in Z$ such that
  $$J(f_0+f_T) \cap (Z^\perp + 0) \neq \emptyset$$
  But $Z=\{L_i\}_\perp$ and so $Z^\perp = \vecspan\{L_i\}$. Thus there exists $\hat{f}=f_0+f_T$ which satisfies the interpolation constraints and such that
  $$J(\hat{f})\cap\vecspan\{L_i\} \neq \emptyset$$
  Further for $\hat{L}\in J(\hat{f})\cap Z^\perp$ we have $-f_T\in\ker(\hat{L})$. If $f_0+f_T$ is exposed by $\hat{L}$ then the tangential bound applies and
  $$\Omega(\hat{f}) = \Omega(f_0+f_T) \leq \Omega((f_0 + f_T)  + (-f_T)) = \Omega(f_0)$$
  so $\hat{f}$ is a solution of the regularised interpolation problem.\\
  If on the  other hand $f_0+f_T$ is not exposed by $\hat{L}$, then it is contained in a face exposed by $\hat{L}$. But then for any $\overline{f_T}\in\banachspace$ such that $\hat{f}+\overline{f_T}$ is still contained in this face we have that $\hat{L}\in J(f_0+f_T+\overline{f_T})$ and $\overline{f_T}\in\ker(\hat{L})$ so that $f_0+f_T+\overline{f_T}$ satisfies the interpolation constraints. We can thus choose $\overline{f_T}$ such that $f_0+f_T+\overline{f_T}$ is a minimum of $\Omega$ in the face and the tangential bound hence applies to it. Thus similarly to before
  $$\Omega(f_0+f_T+\overline{f_T}) \leq \Omega((f_0 + f_T+\overline{f_T})  + (-f_T-\overline{f_T})) = \Omega(f_0)$$
  and $f_0+f_T+\overline{f_T}$ is a solution of the regularised interpolation problem of the desired form.
  \end{subproof}
\end{proof}
This illustrates why strict convexity is the crucial property determining the type of result we can obtain. If the space is strictly convex then every point is rotund and thus exposed. This means every point is admissible and we are in a situation similar to before. We are thus first going to discuss this case, before looking at what can be said when the space is not strictly convex.

\subsection{Strictly Convex Spaces}\label{sec:strictly-convex}
Since in a strictly convex space every point is exposed, every point is admissible and the tangential bound from \cref{lma:tangential_bound} applies everywhere. We thus are able to obtain results in exactly the spirit of our previous work~\cite{schlegel2019}.

\begin{lemma}
  \label{lma:circular_bound}
  If for every $f\in\banachspace$ and all $f_T\in\bigcup\limits_{L\in J(f)}\ker(L)$ we have $\Omega(f) \leq \Omega(f+f_T)$ then for any fixed $\hat{f}\in\banachspace$ we have that
  $$\Omega(\hat{f})\leq\Omega(f)$$
  for all $f\in\banachspace$ such that $\norm{\hat{f}} < \norm{f}$.
\end{lemma}

\begin{proof}
  Since the space is assumed to be strictly convex every point is exposed. The space may not be smooth in which case the duality mapping $J$ is not univocal but for a non-smooth, rotund point $f$ every $L\in J(f)$ exposes it. Thus \cref{lma:tangential_bound} applies to all points $f\in\banachspace$ and all functionals  $L\in J(f)$. We thus do not need to worry about whether or not a point is an exposed point and whether it is exposed by a given functional attaining its norm at the point. This means we can follow the same general idea of argumentation as we did in our previous work~\cite{schlegel2019}.\\
  \begin{subproof}[Bound $\Omega$ on the half spaces given by the tangent planes through $\hat{f}$]
    We start by showing that $\Omega$ is radially nondecreasing by moving out along a tangent and back along another tangent to hit any point along the ray $\lambda\cdot\hat{f}$ for $\lambda>1$. Via the tangents at those points this again immediately gives the bound for all half spaces spanned by a tangent plane through $\hat{f}$ given by some $L\in J(\hat{f})$, which might be more than one with $J$ possibly not being univocal. This is illustrated in \cref{fig:radially_increasing}.\\
    \begin{figure}[h]
      \begin{subfigure}[t]{.5\textwidth}
        \centering
        \includegraphics[width=.95\linewidth]{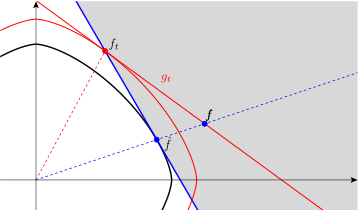}
        \caption{\footnotesize We can extend the tangential bound to the ray $\lambda\cdot \hat{f}$ by finding the point $f_t$ along the tangent from where the tangent to $f_t$ hits the desired point $f=\lambda\cdot\hat{f}$ on the ray. Via the tangents to points along the ray the bound then extends to the shaded half space.}\label{fig:radially_increasing}
      \end{subfigure}
      \begin{subfigure}[t]{.5\textwidth}
        \centering
        \includegraphics[width=.95\linewidth]{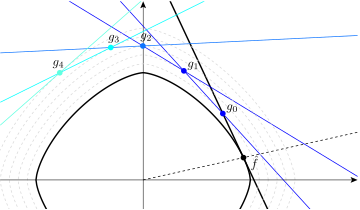}
        \caption{\footnotesize By repeatedly taking small steps along tangents we can move all the way around the circle.}\label{fig:radiating_bound}
      \end{subfigure}
      \caption{Extending the tangential bound to all points outside the circle.}
    \end{figure}\\
    \noindent
    We fix some $\hat{f}\in\banachspace$ and $\lambda>1$ and set $f=\lambda\cdot\hat{f}$. To show that $\Omega(f) \geq \Omega(\hat{f})$ fix any $L_1\in J(\hat{f})$ and $f_{T}\in\ker\{L_1\}$ and set
    \begin{gather*}
      f_t = \hat{f} + t\cdot f_{T}\\
      g_t = f - f_t = (\lambda - 1)\cdot\hat{f} - t\cdot f_{T}
    \end{gather*}
    so that $f_t + g_t = f$. By the choice of $f_{T}$ we have that $\Omega(\hat{f}) \leq \Omega(f_t)$ and we now need to show that there exists $t_0$ such that there exists $L_{t_0}\in J(f_{t_0})$ such that $g_{t_0}\in\ker\{L_{t_0}\}$. This would mean that $\Omega(f_{t_0})\leq\Omega(f_{t_0}+g_{t_0})=\Omega(f)$ as claimed.\\
    \ \\
    To show that such $t_0$ indeed exists we will consider choices $L_t\in J(f_t)$ for every $t$. Note first that by definition of $g_{t}$
    \begin{align}
      L_{t}(g_{t}) = 0 & \Leftrightarrow (\lambda - 1) L_{t}(\hat{f}) = tL_{t}(f_{T}) \notag\\
                           & \Leftrightarrow \lambda L_{t}(\hat{f}) = L_{t}(\hat{f}) + tL_{t}(f_{T}) = L_{t}(f_{t}) \notag\\
      \label{eq:equiv_condition} & \Leftrightarrow \lambda L_{t}(\hat{f}) = \norm{f_{t}}^2
    \end{align}
    which gives us an equivalent condition to find a suitable $t_0$.\\
    We now define the set-valued function $F\,\colon [0,\infty)\rightarrow\mathcal{P}(\R)$
    $$F(t) = \{\lambda L_t(\hat{f})\in\R \setseparator L_t\in J(f_t)\}$$
    By \cref{prop:dual_map_basics} $J(f)$ is non-empty, weakly* closed and convex for every $f\in\banachspace$ so the value of $F(t)$ is either a single value or an interval in $\R$.\\
    It is known that if $\banachspace$ is smooth then $J$ is univocal and norm-to-weak* continuous so that $F$ is clearly continuous. We show that if $\banachspace$ is not smooth the function $F$ is still almost continuous in the sense that in any jump the function is interval valued and the interval connects both ends of the jump. To show this fix an arbitrary $t\in[0,\infty)$ and let $s\rightarrow t$. Then $f_s\rightarrow f_t$ in norm and hence for any choice of $L_s\in J(f_s)$ we have that $\norm{L_s}=\norm{f_s}\leq M$ for some constant $M$. Thus passing to a subsequence if necessary $L_s\overset{\ast}{\rightharpoonup} \widetilde{L}$, in particular $L_s(\hat{f})\converges[t]{s}\widetilde{L}(\hat{f})$.\\
    We want to show that this $\widetilde{L}$ is indeed contained in $J(f_t)$. By standard results (c.f. Brezis~\cite{brezis2011} Proposition 3.13 (iv)) we know that $L_s(f_s) \rightarrow \widetilde{L}(f_t)$ but also $L_s(f_s)=\norm{f_s}^2\rightarrow\norm{f_t}^2$ so that
    \begin{equation}
      \label{eq:norm_ft_sq}
      \widetilde{L}(f_t)=\norm{f_t}^2
    \end{equation}
    Further $\norm{\widetilde{L}} \leq \liminf\norm{L_s}=\norm{f_t}$ (c.f. Brezis~\cite{brezis2011} Proposition 3.13 (iii)) and thus
    $$\norm{\widetilde{L}}\cdot\norm{f_t} \leq \liminf\norm{L_s}\cdot\lim\norm{f_s} \leq \lim L_s(f_s) = \widetilde{L}(f_t) \leq \norm{\widetilde{L}}\cdot\norm{f_t}$$
    which means that
    \begin{equation}
      \label{eq:norm_L_ft}
      \widetilde{L}(f_t) = \norm{\widetilde{L}}\cdot\norm{f_t}
    \end{equation}
    Putting \cref{eq:norm_ft_sq} and \cref{eq:norm_L_ft} together gives
    $$\norm{f_t}^2=\widetilde{L}(f_t)=\norm{\widetilde{L}}\cdot\norm{f_t}$$
    which shows that indeed $\norm{\widetilde{L}} = \norm{f_t}$ and hence $\widetilde{L}\in J(f_t)$.\\
    But this means that for $s\rightarrow t$ and any choice of $F(s)$ where $F$ is not single valued there exists $x\in F(t)$ such that $F(s)\rightarrow x$. This proves the claim that $F$ is ``effectively continuous'', in the sense that whenever the function would have a jump it is set valued and its interval value closes the gap between either side of the jump. This means that an intermediate value theorem holds for the function $F$.\\
    Going back to \cref{eq:equiv_condition} we see that it is satisfied if and only if $\norm{f_{t_0}}^2\in F(t_0)$. For $t=0$, i.e.\  $f_0=\hat{f}$, we have
    $$F(0) = \lambda L_0(\hat{f}) = \lambda\norm{\hat{f}}^2 > \norm{\hat{f}}^2 = \norm{f_0}^2$$
    On the other hand
    $$F(t) = \lambda L_t(\hat{f}) \leq \lambda\norm{L_t}\cdot\norm{\hat{f}} = \lambda\norm{f_t}\cdot\norm{\hat{f}}$$
    But since $\norm{f_t}\converges{t}\infty$ we have $\lambda\norm{\hat{f}} < \norm{f_t}$ for $t$ large enough and thus
    $$F(t) = \lambda L_t(\hat{f}) \leq \lambda \norm{f_t}\cdot\norm{\hat{f} } < \norm{f_t}^2$$
    for large $t$. Since $\norm{f_t}^2$ is continuous in $t$ and the intermediate value theorem holds for $F$ this means that there exists a $t_0$ such that $\norm{f_{t_0}}^2 \in F(t_0)$ which means that there exists $L_{t_0}\in J(f_{t_0})$ such that \cref{eq:equiv_condition} is satisfied. For this $t_0$ indeed
    $$\Omega(\hat{f}) \leq \Omega(f_{t_0}) \leq \Omega(f_{t_0} + g_{t_0}) = \Omega(f)$$
  \end{subproof}

  \begin{subproof}[Extend the bound around the circle]
    The fact that we can extend the bound around the circle is clear by the same argument as in our previous work~\cite{schlegel2019}. The idea is that we can repeatedly move along tangents around the circle without moving to far away from it, as illustrated in \cref{fig:radiating_bound}.
    For points of smoothness of the norm we already showed in~\cite{schlegel2019} that if we take small enough steps along tangents we can get all the way around the circle without getting too far away from it. In points of non-smoothness we have more than one tangent to the ball. But as the tangential bound on $\Omega$ holds for every tangent it is obviously always possible to choose a tangent which stays arbitrary close to the circle.
  \end{subproof}
\end{proof}
Seeing that this result is effectively the same as what we proved for uniform Banach spaces in~\cite{schlegel2019} it is not surprising that the main result describing admissible regularisers for strictly convex Banach spaces is the same as for uniform Banach spaces. We can obtain the same closed form characterisation as before, saying that admissible regularisers are almost radially symmetric.

\begin{theorem}
\label{thm:almost_radially_symmetric_reflexive}
  A function $\Omega\,\colon\banachspace\rightarrow\R$ is admissible if and only if it is of the form
  $$\Omega(f) = h(\norm{f}_\banachspace)$$
  for some nondecreasing $h\,\colon [0,\infty)\rightarrow\R$ whenever $\norm{f}_\banachspace\neq r$ for $r\in\mathcal{R}$. Here $\mathcal{R}$ is an at most countable set of radii where $h$ has a jump discontinuity. For any $f$ with $\norm{f}_\banachspace=r\in\mathcal{R}$ the value $\Omega(f)$ is only constrained by the monotonicity property, i.e.\ it has to lie in between $\lim\limits_{t\nearrow r}h(t)$ and $\lim\limits_{t\searrow r}h(t)$.
\end{theorem}

The proof given in~\cite{schlegel2019} for the analogue of this theorem (theorem 3.2 in~\cite{schlegel2019}) is in fact still entirely valid. We thus only comment briefly on a few important points. Note in particular that from the fact that for any  $f_T\in\bigcup\limits_{L\in J(f)}\ker(L)$ we have
  $$L(f+f_T)=L(f)=\norm{L}\cdot\norm{f}\leq\norm{L}\cdot\norm{f+f_T}$$
  and so $\norm{f}\leq\norm{f+f_T}$. By strict convexity the inequality is in fact strict so that the bound for the mollification in part 2 of theorem 3.2 in~\cite{schlegel2019} remains valid. It is also clear that \cref{spf:tangential_bound_part1} of the proof of \cref{lma:tangential_bound} holds for $f=0$ so 0 is an admissible point. Thus $\Omega$ is without loss of generality minimised at 0 with $\Omega(0)=0$. All other parts of the proof of theorem 3.2 are also clearly still valid.

  \subsection{Non-strictly convex spaces and $\leb{1}$}
  Obtaining a general, closed form geometric interpretation of the tangential bound as we presented above is very difficult for spaces which are not strictly convex. This is due to the large geometric variety of Banach spaces, making it very hard to make any statements about the shape of the unit ball, even locally. We can e.g.\ construct a Banach space with a rotund point such that no point in its neighbourhood is rotund. Similarly a convex function on $\R$ may not be differentiable on a countable dense subset (c.f.\ e.g.~\cite{lindenstrauss2012,phelps1993}) so also smoothness does not allow statements about surrounding points. Worse even, there might not even be any exposed point, e.g.\ the space $c_0$ does not contain exposed points. We are thus going to restrict our intention to the space that is most commonly used in applications, $\leb{1}_n$. The space $\leb{1}$ is only reflexive if it is finite dimensional but in applications we are often going to do computations in a finite truncation of $\leb{1}$ so that this is an interesting case to consider.\\
  Fixing the space to be a concrete example does remove the issue of geometric variety and it turns out that this allows to run an argument similar to the one presented in \cref{sec:strictly-convex}.

 \begin{lemma}\label{lma:circular_bound_l1}
  If for every exposed face of the norm ball in $\leb{1}_n$ $\Omega$ attains its minimum in at least one point and for every $f$ in the face where the minimum is attained and every $L\in J(f)$ exposing the face and every $f_T\in\ker(L)$ we have $\Omega(f+f_T) \geq \Omega(f)$ then for any fixed admissible $\hat{f}\in\leb{1}_n$ we have that
  $$\Omega(\hat{f})\leq\Omega(f)$$
  for all $f\in\leb{1}_n$ such that $\norm{\hat{f}} < \norm{f}$.
\end{lemma}

\begin{proof}
  \begin{figure}[b]
      \begin{subfigure}[t]{.5\textwidth}
        \centering
        \includegraphics[width=.95\linewidth]{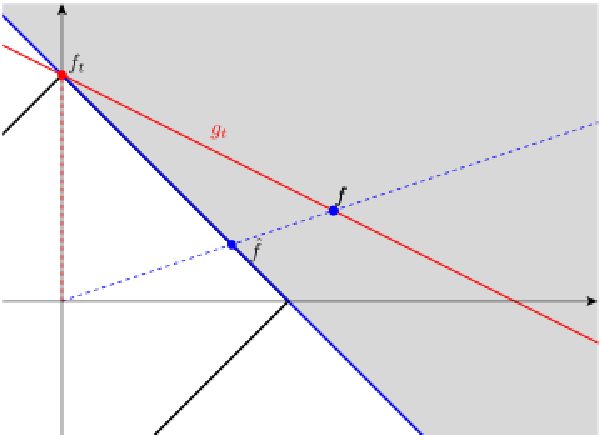}
        \caption{\footnotesize We can go from an admissible point in a face $F$ to a vertex and from there bound the minima in the face $\lambda F$ for $1<\lambda$.}\label{fig:radially_increasing_l1}
      \end{subfigure}
      \begin{subfigure}[t]{.5\textwidth}
        \centering
        \includegraphics[width=.95\linewidth]{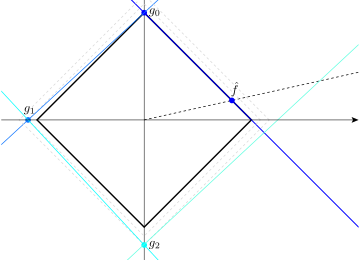}
        \caption{\footnotesize By repeatedly taking steps along tangents we can move all the way around the unit ball without moving far away from it.}\label{fig:radiating_bound_l1}
      \end{subfigure}
      \caption{Extending the tangential bound to all points outside the circle.}
    \end{figure}
    \noindent
  \begin{subproof}[Bound $\Omega$ on the half spaces given by the tangent planes through $\hat{f}$]\label{spf:radially_nondecreasing}
    As in \cref{sec:strictly-convex} we first show that $\Omega$ is radially nondecreasing. Notice that in $\leb{1}_n$ every vertex $e_i = (0,\ldots,0,1,0,\ldots,0)$ of the unit ball is an exposed point and hence admissible. If we fix an admissible $\hat{f}\in\leb{1}_n$ then it is either one of the vertices $e_i$, or a minimum within a face which is the convex hull of several $e_i$.\\
    \ \\
  If $\hat{f}$ is one of the vertices, $e_k$ say, then it is clear that there exists a linear functional $\hat{L}$ exposing it so that along the tangent given by $\hat{L}$ we can reach a different vertex $\lambda_0 e_j$ for $\lambda_0>1$ say. Now by the same argument we can find a tangent in the reverse direction, connecting $\lambda_0 e_j$ to $\lambda_1 e_k$, $\lambda_1>\lambda_0$. It is clear that we can control the size of $\lambda_0$ and $\lambda_1$ to hit the desired point $\lambda\hat{f}$ for $1<\lambda$.\\
  If on the other hand $\hat{f}$ is the minimum within a face $F$ exposed by the linear functional $\hat{L}$ then $\hat{f}$ satisfies the tangential bound for $\hat{L}$. Thus from $\hat{f}$ we can reach any vertex, $e_k$ say, on the boundary of the face $F$ along a tangent given by $\hat{L}$. It is clear that $e_k$, being an exposed point point, has a tangent plane which is close to the face $F$, so that we can reach the minimum in the face $\lambda F$ for $1<\lambda$. Note that this minimum might not be $\lambda\hat{f}$. This is illustrated in \cref{fig:radially_increasing_l1}.\\
  Combining both arguments we see that the minimum of $\Omega$ within a given face is a nondecreasing function of the norm. Clearly with the tangent planes of the minima we get the same bound for any half space spanned by a tangent plane at $\hat{f}$ as in \cref{sec:strictly-convex}.
\end{subproof}

\begin{subproof}[Extend the bound around the circle]
  The fact that we can extend the bound around the circle in the same way as previously is clear from the arguments in \cref{spf:radially_nondecreasing}. We already noticed that if $\hat{f}$ is within a face we can reach any vertex on the boundary of the face. We further know that from a vertex we can get across any face containing it to another vertex while staying arbitrarily close to the face, as illustrated in \cref{fig:radiating_bound_l1}. Hence it is clear that we can reach any admissible $\overline{f}$ with $\norm{\overline{f}}>\norm{\hat{f}}$.\\
  Putting both observations together we get the claim.
\end{subproof}
\end{proof}

This proof illustrates that in the case of $\leb{1}_n$ it may be convenient to view $\Omega$ as a function of the faces of the norm ball. In other words we are thinking of the faces as being collapsed to one point where $\Omega$ is minimised. Viewed as a function of the faces $\Omega$ is indeed almost radially symmetric again.

\begin{theorem}
  \label{thm:almost_radially_symmetric_non_reflexive}
  A function $\Omega\,\colon\leb{1}_n\rightarrow\R$ is admissible if and only if viewed as a function $\overline{\Omega}$ of the faces $F$ of the norm ball in $\leb{1}_n$
  $$\overline{\Omega}(F) = \min\limits_{f\in F}\Omega(f)$$
  it is of the form
  $$\overline{\Omega}(F) = h(\norm{f}_{\leb{1}_n} \setseparator f\in F)$$
  for some nondecreasing $h\,\colon [0,\infty)\rightarrow\R$ whenever $\norm{f}_{\leb{1}_n}\neq r$ for $r\in\mathcal{R}$. Here $\mathcal{R}$ is an at most countable set of radii where $h$ has a jump discontinuity. For any $f$ with $\norm{f}_{\leb{1}_n}=r\in\mathcal{R}$ the value $\overline{\Omega}(F)$ is only constrained by the monotonicity property., i.e.\ it has to lie in between $\lim\limits_{t\nearrow r}h(t)$ and $\lim\limits_{t\searrow r}h(t)$.\\
  \ \\
  Moreover in points of continuity of $h$ the function $\Omega$ attains its minimum in a face $F$ in every exposed point within the face.
\end{theorem}

\begin{proof}
  The proofs for uniform Banach spaces and strictly convex Banach spaces remain largely valid, only few extra considerations are required. We are going to briefly discuss sections which remain valid and present in full any extra arguments which are required.\\
  Firstly, the fact that continuity in radial direction implies radial symmetry of $\overline{\Omega}$ is clear since we only need to consider admissible points and for two admissible points $f$ and $g$ the previous argument obviously holds.\\
  For this observation to be useful we need to verify again that the radially mollified regulariser
    $$\widetilde{\Omega}(f) = \int\limits_{-1}^0 \rho(t)\Omega\left((\norm{f}-t)\frac{f}{\norm{f}}\right)\differential{t}$$
    is admissible if $\Omega$ was admissible.\\
    More precisely we check that $\widetilde{\Omega}$ is still non-decreasing along tangential directions, i.e.~we need to show that for an admissible $f$ for any $L\in J(f)$ exposing the face containing $f$ and every $f_T\in\ker(L)$ we still have
    \begin{multline}
      \label{eq:integral_estimate_non_rotund}
      \widetilde{\Omega}(f+f_T) = \int\limits_{-1}^0 \rho(t)\Omega\left((\norm{f+f_T}-t)\frac{f+f_T}{\norm{f+f_T}}\right)\differential{t}\\
      \geq \int\limits_{-1}^0 \rho(t)\Omega\left(\left(\norm{f}-t\right)\frac{f}{\norm{f}}\right)\differential{t} = \widetilde{\Omega}(f)
    \end{multline}
    The previous proof was  based on the fact that $\norm{f+f_T} > \norm{f}$. Whenever this is true the proof holds, so we only need to check the case when $\norm{f+f_T}=\norm{f}$. But in this case we have that
  $$\Omega\left((\norm{f+f_T}-t)\frac{f+f_T}{\norm{f+f_T}}\right) = \Omega\left((\norm{f}-t)\frac{f+f_T}{\norm{f}}\right) = \Omega\left(\frac{\norm{f}-t}{\norm{f}}(f+f_T)\right)$$
  Since $\frac{\norm{f}-t}{\norm{f}}f_T\in\ker(L)$ we have that indeed
  $$\Omega\left((\norm{f+f_T}-t)\frac{f+f_T}{\norm{f+f_T}}\right) \geq \Omega\left((\norm{f}-t)\frac{f}{\norm{f}}\right)$$
  Since $\rho$ is positive that means the integrand of $\widetilde{\Omega}(f+f_T)$ is greater or equal than the integrand of $\widetilde{\Omega}(f)$ so that the property of being nondecreasing along all tangents is indeed preserved.\\
  \ \\
  Putting these two observations together we obtain the result. We know that as a function of the faces $\overline{\Omega}$ is a monotone function of the norm, so a monotone function on the real line. After mollification $\overline{\Omega}$ is in fact radially symmetric. The same considerations as before say that $\overline{\Omega}$ must have been of the claimed form.\\
  \ \\
  The converse is clear since the value of $\overline{\Omega}$ is defined to be the minimum across each face, so minima exist and clearly satisfy the tangential bound.\\
  \ \\
  For the moreover part let $F$ be a face of the unit ball and $g$ a minimum of $\Omega$ in $F$. Assume further that $h$ is continuous in $\norm{g}$. Fix a vertex $e_j$ in $F$. Then there clearly exists a tangent from $\lambda e_j$ to $g$ for $1-\varepsilon < \lambda < 1$ and thus $\Omega(\lambda e_j) \leq \Omega(g)$. By continuity of $h$ in $\norm{g}=\norm{e_j}$ we have $\Omega(\lambda e_j)\converges[1]{\lambda}\Omega(e_j)$ and so $\Omega(e_j) \leq \Omega(g)$. Since $g$ is the minimum in $g$ this means $\Omega(e_j)=\Omega(g)$.
\end{proof}

This shows that a very similar intuition to the results for strictly convex spaces indeed is true for $\leb{1}_n$. Moreover it shows that any admissible regulariser in $\leb{1}_n$ will attain its minimum at the vertices which is exactly the reason for their use in applications.
\section{Consequences and Optimality}\label{sec:consequences}
Just as presented in the authors work~\cite{schlegel2019} an important consequence of above results is that, if one relies on the representer theorem for learning, in fact the solution of the regularised interpolation problem in most cases does not depend on the regulariser but is determined by the function space alone. This has two important consequences. Firstly it means we are free to work with whatever regulariser is most convenient for our purpose, whether this is computational applications or proving theoretical results. Secondly it illustrates the importance of extending well established learning methods for Hilbert spaces to Banach spaces to allow for a greater variety of spaces to learn in.\\
In this section we discuss this fact and also illustrate why one can not hope to weaken the assumption on the function space any further than reflexivity.
\subsection{The solution is determined by the space}
Throughout this section we say that a function $f_0$ is a representer theorem solution of (\ref{eq:interpolation_problem}) if it is a solution of (\ref{eq:interpolation_problem}) in the sense of \cref{def:admissibility}, i.e.\ such that there exists $\hat{L} = \sum\limits_{i=1}^m c_i L_i$ such that $\hat{L}\in J(f_0)$. To prove above claim that the solution is often independent of the regulariser we are going to show that in most cases a function $f_0$ is a representer theorem solution of (\ref{eq:interpolation_problem}) if and only if it is a solution of the minimal norm interpolation problem
\begin{equation}
  \label{eq:minimal_norm_interpolation}
  \inf\left\{\norm{f} \setseparator f\in\banachspace,\, L_i(f)=y_i\,\forall i\in\N_m\right\}
\end{equation}
This follows by combining above results with a result by Micchelli and Pontil from~\cite{miccelli2004}. They consider the  minimal norm interpolation problem
\begin{equation}
  \label{eq:minimal_norm_interpolation}
  \inf\{\norm{f}_X \setseparator f\in X, L_i(f) = y_i\,\forall i\in\N_m\}
\end{equation}
for a general Banach space $X$. Under the assumption that $X$ is reflexive they prove a necessary and sufficient condition for a function to be a solution of this problem.
\begin{proposition}[Theorem 1 in~\cite{miccelli2004}]\label{prop:minimal_norm_solution}
Let $X$ be reflexive. $f_0$ is a solution of~\cref{eq:minimal_norm_interpolation} if and only if it satisfies the constraints $L_i(f_0)=y_i$ and there is a linear combination of the continuous linear functionals defining the problem which peaks at $f_0$, i.e.~there exists $(c_1,\ldots,c_m)\in\R^m$ such that
$$\sum\limits_{i=1}^m c_i L_i(f_0) = \norm{\sum\limits_{i=1}^m c_i L_i}_{\banachspace^\ast}\cdot\norm{f_0}_\banachspace$$
\end{proposition}
This corresponds to $h(t)=t$ in \cref{eq:interpolation_problem}. We now get the following theorem.
\begin{theorem}\label{thm:space_dependence}
Let $\banachspace$ be a reflexive Banach space and $\Omega$ admissible. Then any representer theorem solution of~(\ref{eq:interpolation_problem}) is a solution of~(\ref{eq:minimal_norm_interpolation}).\\
Moreover for any solution of~(\ref{eq:minimal_norm_interpolation}) there exists a representer theorem solution of~(\ref{eq:interpolation_problem}) in the same face of the norm ball. Thus in particular if $\banachspace$ is strictly convex then $f_0$ is is a representer theorem solution of~(\ref{eq:interpolation_problem}) if and only if it is a solution of~(\ref{eq:minimal_norm_interpolation}).
\end{theorem}

\begin{proof}
  \begin{subproof}[A solution of~(\ref{eq:interpolation_problem}) is a solution of~(\ref{eq:minimal_norm_interpolation})]
    Assume that $f_0$ is a representer theorem solution of~(\ref{eq:interpolation_problem}). Then since\\
    $\vecspan\{L_i\setseparator i\in\N_m\}={(\vecspan{\{L_i\setseparator i\in\N_m\}}_\perp)}^\perp$ we have for any  $\hat{L} = \sum\limits_{i=1}^m c_iL_i \in J(f_0)$ and all $f_T\in\vecspan\{L_i\setseparator i\in\N_m\}_\perp$ that
    $$\norm{\hat{L}}_{\banachspace^\ast}\cdot\norm{f_0}_\banachspace = \hat{L}(f_0) = \hat{L}(f_0+f_T) \leq \norm{\hat{L}}_{\banachspace^\ast}\cdot\norm{f_0+f_T}$$
    so $\norm{f_0}_\banachspace\leq\norm{f_0+f_T}_\banachspace$ and $f_0$ is a solution of (\ref{eq:minimal_norm_interpolation}).
  \end{subproof}

  \begin{subproof}[For any sol. of~(\ref{eq:minimal_norm_interpolation}) $\exists$ a sol. of~(\ref{eq:interpolation_problem}) in the same face]
    Assume $f_0$ is a solution of the minimal norm interpolation problem~(\ref{eq:minimal_norm_interpolation}). Then by \cref{prop:minimal_norm_solution} there exists an $\hat{L} = \sum\limits_{i=1}^m c_i L_i$ such that $\hat{L}(f_0) = \norm{\hat{L}}_{\banachspace^\ast}\cdot\norm{f_0}_\banachspace$ and thus $\frac{\norm{f_0}_\banachspace}{\norm{\hat{L}}_{\banachspace^\ast}}\hat{L}\in J(f_0)$.\\
    Further if $f_0$ is an admissible point in the sense of \cref{def:admissibility}, then the tangential bound \cref{lma:tangential_bound} applies and
    $$\Omega(f_0)\leq\Omega(f_0+f_T) \qquad \forall f_T\in\vecspan\{L_i\setseparator i\in\N_m\}_\perp$$
    so $f_0$ is a representer theorem solution of (\ref{eq:interpolation_problem}).\\
    If $f_0$ is not admissible in the sense of \cref{def:admissibility} then there exists an admissible point $\overline{f_0}$ in the same face for which above inequality holds so that $\overline{f_0}$ is a representer theorem solution of (\ref{eq:interpolation_problem}).\\
    \ \\
    If $\banachspace$ is strictly convex then every point is admissible and $f_0$ is is a representer theorem solution of~(\ref{eq:interpolation_problem}) if and only if it is a solution of~(\ref{eq:minimal_norm_interpolation}).
  \end{subproof}
\end{proof}

This result shows that for any admissible regulariser on a reflexive, strictly convex Banach space the set of solutions with a dual element in the linear span of the defining linear functionals is identical. This in particular means that it is the choice of the function space, and only the choice of the space, which determines the solution of the problem. We are thus free to work with whichever regulariser is most convenient in application. Computationally in many cases this is likely going to be $\frac{1}{2}\norm{\cdot}^2$. For theoretical results other regularisers may be more suitable, such as in the aforementioned paper~\cite{miccelli2004} which heavily relies on a duality between the norm of the space and its continuous linear functionals.\\
For a reflexive Banach space which is not strictly convex the solution is also mostly determined by the space, the regulariser only determines the point(s) within a certain face of the norm ball which is optimal. The face containing the solution again is independent of $\Omega$.

\subsection{Reflexivity is necessary}
The fact that \Cref{prop:minimal_norm_solution} is an if and only if suggests that one can not do better than reflexivity in the assumptions on the space without weakening other assumptions. And indeed this is the case. The duality mapping $J$ is surjective if and only if the space $X$ is reflexive. Thus in a non-reflexive Banach space we can find $L_i$ which are not the image of any element in $X$ under the duality mapping. In this case there is no hope of finding a solution in the sense of \Cref{def:admissibility}.\\
As an example consider $X=l_1$ with $X^\ast=l_\infty$. Let $L_1=(x_i)_{i\in\N}$ where $x_i=\frac{i}{i+1}$ for $i$ odd and $x_i=0$ for $i$ even and $L_2=(y_i)_{i\in\N}$ where $y_i=\frac{i}{i+1}$ for $i$ even and $y_i=0$ for $i$ odd, i.e.
$$L_1=(\frac12,0,\frac34,0,\ldots) \mbox{ and } L_2=(0,\frac23,0,\frac45,\ldots)$$
Then $\norm{L_1} = \norm{L_2} = 1$ but there cannot be a $l_1$-sequence of norm 1, $x$ say, such that $L_1(x)=1$ or $L_2(x)=1$. So $L_1,L_2\not\in J(X)$. It is also clear by construction that the same is true for linear combinations of $L_1$ and $L_2$ so
$$\vecspan\{L_1,L_2\}\cap J(X) = \{0\}$$
This means there is no hope of finding a solution in the sense of \cref{def:admissibility} with a dual element in the linear span of the defining linear functionals.

\section{Examples}\label{sec:examples}
In this section we give several examples of Banach spaces to which the results in this paper apply. These examples are taken from the work of Zhang, Xu and Zhang\cite{zhang2009}, and Zhang and Zhang\cite{zhang2012}.

In these papers the theory of reproducing kernel Banach spaces (RKBS) is developed. This generalises the very well known theory of reproducing kernel Hilbert spaces, providing several advantages which we will discuss throughout this section. We begin by stating some of the key definitions and results for RKBS.\\
\ \\
We call a Banach space $\banachspace$ a Banach space of functions if for each $f\in\banachspace$ its norm $\norm{f}_\banachspace$ vanishes if and only if $f$ is identically zero.
\begin{definition}[Reproducing Kernel Banach Space]
  A Banach space of functions $\banachspace$ is a reproducing kernel Banach space (RKBS) if it is reflexive, its dual space $\banachspace^\ast$ is isometric to a Banach space of functions $\banachspace^\sharp$ and point evaluations are continuous on both $\banachspace$ and $\banachspace^\sharp$.
\end{definition}
It is convenient to view an element $f^\ast\in\banachspace^\ast$ as a function on $X$ by identifying it via the isometry with $f^\sharp\in\banachspace^\sharp$ and simply writing $f^\ast(x)$.\\
With this definition one obtains a theorem reminiscent of reproducing kernel Hilbert spaces.
\begin{theorem}
  Let $\banachspace$ be a RKBS on a set $X$. Then there exists a unique function\\
  $K:X\times X\rightarrow\R$ such that
  \begin{enumerate}[label=(\emph{\alph*})]
  \item For every $x\in X$, $K(\cdot, x)\in\banachspace^\ast$ and $f(x) = \dualmap{\banachspace}{f}{K(\cdot,x)}$ for all $f\in\banachspace$.
    \item For every $x\in X$, $K(x,\cdot)\in\banachspace$ and $f^\ast(x) = \dualmap{\banachspace}{K(x,\cdot)}{f^\ast}$ for all $f^\ast\in\banachspace^\ast$.
    \item $\overline{\vecspan}\{K(x,\cdot) \setseparator x\in X\} = \banachspace$
    \item $\overline{\vecspan}\{K(\cdot,x) \setseparator x\in X\} = \banachspace^\ast$
    \item For all $x,y\in X$ $K(x,y) = \dualmap{\banachspace}{K(x,\cdot)}{K(\cdot,y)}$
  \end{enumerate}
\end{theorem}
It now turns out that there is a convenient way of constructing reproducing kernel Banach spaces.
\begin{theorem}
  Let $\mathcal{W}$ be a reflexive Banach space with dual space $\mathcal{W}^\ast$ and let $\Phi:X\rightarrow\mathcal{W}$ and $\Phi^\ast:X\rightarrow\mathcal{W}^\ast$ maps such that $\overline{\vecspan}\Phi(X) = \mathcal{W}$ and $\overline{\vecspan}\Phi^\ast(X) = \mathcal{W}^\ast$. Then there exists a RKBS $\banachspace$ which is isometrically isomorphic to $\mathcal{W}$ given by
  $$\banachspace = \{\dualmap{\mathcal{W}}{u}{\Phi^\ast(\cdot)} \setseparator u\in\mathcal{W}\}\quad\mbox{ with norm }\quad\norm{\dualmap{\mathcal{W}}{u}{\Phi^\ast(\cdot)}}_\banachspace = \norm{u}_{\mathcal{W}}$$
  with dual space $\banachspace^\ast$ which is isometrically isomorphic to $\mathcal{W}^\ast$ and given by
  $$\banachspace^\ast = \{\dualmap{\mathcal{W}}{\Phi(\cdot)}{u^\ast} \setseparator u^\ast\in\mathcal{W}^\ast\}\quad\mbox{ with norm }\quad\norm{\dualmap{\mathcal{W}}{\Phi(\cdot)}{u^\ast}}_{\banachspace^\ast} = \norm{u^\ast}_{\mathcal{W}^\ast}$$
  The reproducing kernel is given by $K(x,y) = \dualmap{\mathcal{W}}{\Phi(x)}{\Phi^\ast(y)}$.
\end{theorem}
As an example of these constructions consider the following example given by Zhang, Xu and Zhang \cite{zhang2009}. Let $X=\R$ and $\mathcal{W}=L^p(\mathbb{I})$ with $\mathbb{I}=[-\frac12,\frac12]$. With
$$\Phi(x)(t) = e^{-2\pi ixt}, \, \Phi^\ast(x)(t) = e^{2\pi ixt}, \quad x\in\R,t\in\mathbb{I}$$
we obtain a RKBS
$$\banachspace=\{f\in C(\R) \setseparator \supp\hat{f}\subseteq\mathbb{I},\hat{f}\in L^p(\mathbb{I})\}$$
with dual space
$$\banachspace^\ast=\{g\in C(\R) \setseparator \supp\check{g}\subseteq\mathbb{I},\check{g}\in L^q(\mathbb{I})\}$$
and kernel
$$K(x,y) = \dualmap{L^p(\mathbb{I})}{\Phi(x)}{\Phi^\ast(y)} = \frac{\sin\pi(x-y)}{\pi(x-y)}=\mbox{sinc}(x-y)$$
The duality pairing is given by
$$\dualmap{\banachspace}{f}{g} = \int\limits_{\mathbb{I}}\hat{f}(t)\check{g}(t)\differential{t} \quad f\in\banachspace, g\in\banachspace^\ast$$
For $p=q=2$ this construction corresponds to the usual space of bandlimited functions. For other values of $p$ we maintain the property of a Fourier transform with bounded support but consider a different $L^p$ norm making $\banachspace$ isometrically isomorphic to $L^p(\mathbb{I})$.\\
Since unlike Hilbert spaces of the same dimension the $L^p(\mathbb{I})$ spaces are not isomorphic to each other they exhibit a richer geometric variety which is potentially useful for the development of new learning algorithms.\\
Note that above example is one dimensional for notational simplicity and similar constructions yield RKBS isomorphic to $L^p_\mu(\R^d)$ where $\mu$ is a finite positive Borel measure on $\R^d$ as shown in Zhang and Zhang \cite{zhang2012}. The corresponding RKBS $\banachspace$ consists of functions of the form
$$f_u(x) = \frac{1}{{\mu(\R^d)}^{\frac{p-2}{p}}}\int\limits_{\R^d}u(t)e^{i\inner[]{x}{t}}\differential{\mu(t)}, \quad x\in\R^d, u\in L^P_\mu(\R^d)$$
and the reproducing kernel is given by
$$K(x,y) = \frac{1}{{\mu(\R^d)}^{\frac{p-2}{p}}}\int\limits_{\R^d}e^{i\inner[]{y-x}{t}}\differential{\mu(t)},\quad x,y\in\R^d$$
For $d=1$ and $\mu$ the Lebesgue measure this reduces to the above example.\\
The dual map in $L^p$ spaces is given by $f^\ast = \frac{f\modulus{f}^{p-2}}{\norm{f}^{p-2}_p}$ which in the given example means that for an element $f_u\in\banachspace$ the corresponding dual element is given by
$$f_u^\ast = \frac{u\cdot\modulus{u}^{p-2}}{\norm{u}_p^{p-2}}$$
Further the dual map in a reflexive Banach space is self-inverse so
$$(f_u^\ast)^\ast = f_u$$
\ \\
These constructions are of interest for various reasons. Firstly this allows us to learn in a larger variety of function spaces which may be of use if we are expecting the solution in a certain class due to prior knowledge, or if we fail to find a good enough solution in a Hilbert space, or if the data has some intrinsic structure that makes it impossible to embed into a Hilbert space.\\
Furthermore, as in contrast to Hilbert spaces two Banach spaces of the same dimension need not be isometrically isomorphic, Banach spaces exhibit a much richer geometric variety which is potentially useful for developing new learning algorithms.
\ \\
Secondly it is often desirable to use norms which are not induced by an inner product because they possess useful properties for application. It is often stated in the literature that a regulariser is used to enforce a certain property such as sparsity or smoothness. But as we showed in \cref{sec:consequences} it is in fact not the regulariser as such but the norm of the function space alone which provides any desired property.\\
As an example consider $L^1$ regularisation which is often used to induce sparsity of the solution. Sparsity occurs because in $L^1$ all extreme points of the unit ball lie on the coordinate axes. The finite dimensional spaces $l^1_d$ are reflexive and thus fall into the framework of this paper. The infinite dimensional spaces $l^1$ and $L^1$ are not reflexive but one can instead work in a $L^p$ space for $p$ close to 1, see e.g.\ Tropp~\cite{tropp2006}.\\
\ \\
\begin{appendix}
\bibliographystyle{acm}
\bibliography{representer.bib}
\pagebreak
\section{Appendix}

\subsection{Regularised Interpolation}
The proof of \cref{thm:problem_equivalence} is largely identical to the one presented in Argyriou, Micchelli and Pontil~\cite{argyriou2009} but requires a few minor adjustments to hold for the generality of reflexive Banach spaces. We present the full proof here.\\

\begin{proof}[Of \Cref{thm:problem_equivalence}]
  To prove that $\Omega$ is admissible for the regularised interpolation problem (\ref{eq:interpolation_problem}) we are going to show that $\Omega$ is tangentially nondecreasing in the sense of \cref{lma:tangential_bound} depending on the properties of the space $\banachspace$.\\
  Fix $0 \neq f\in\banachspace$ and $L\in J(f)$ and let $a_0$ be the unique nonzero minimiser of $\min\{\error(a\nu,y) \setseparator a\in\R\}$. For every $\lambda>0$ Consider the regularisation problem
  $$\min\left\{\error\left(\frac{a_0}{\norm{L}^2}L(f)\nu,y\right) + \lambda\Omega(f) \setseparator f\in\banachspace\right\}$$
  By assumption there exist solutions $f_\lambda\in\banachspace$ such that
  $$J(f_\lambda)\cap\vecspan{L}\neq\emptyset$$
  i.e.\ there exist $c_\lambda\in\R$ such that $c_\lambda L\in J(f_\lambda)$.\\
  Now fix any $g\in\banachspace$ such that $L\in J(g)$ which exists as $\banachspace$ is reflexive so $J$ is surjective. We then obtain
  \begin{equation}
    \label{eq:ineq_omega}
    \error(a_0\nu,y)+\lambda\Omega(f_\lambda) \leq \error\left(\frac{a_0}{\norm{L}^2}L(f_\lambda)\nu,y\right)+\lambda\Omega(f_\lambda) \leq \error(a_0\nu,y)+\lambda\Omega(g)
  \end{equation}
  where the first inequality follows from $a_0$ minimising $\error(a\nu,y)$ and the second inequality from $L(g)=\norm{L}^2$. This shows that $\Omega(f_\lambda)\leq\Omega(g)$ for all $\lambda$ and so by hypothesis the set $\{f_\lambda \setseparator \lambda>0\}$ is bounded. Hence there exists a weakly convergent subsequence $(f_{\lambda_l})_{l\in\N}$ such that $\lambda_l\converges{l}0$ and $f_{\lambda_l}\rightharpoonup\overline{f}$ as $l\rightarrow\infty$. Taking the limit inferior as $l\rightarrow\infty$ on the right hand side of inequality \cref{eq:ineq_omega} we obtain
  $$\error\left(\frac{a_0}{\norm{L}^2}L(\overline{f})\nu,y\right) \leq \error(a_0\nu,y)$$
  Since $a_0$ is by assumption the unique, nonzero minimiser this means that
  $$\frac{a_0}{\norm{L}^2}L(\overline{f}) = a_0 \Leftrightarrow L(\overline{f}) = \norm{L}^2$$
  But then since $L(\overline{f}) \leq \norm{L}\cdot\norm{\overline{f}}$ we have $\norm{L}\leq\norm{\overline{f}}$.\\
  Moreover since $J(f_\lambda)\cap\vecspan\{L\}\neq\emptyset$ we have $\norm{L}\cdot\norm{f_\lambda} = L(f_\lambda) \rightarrow \norm{L}^2$ and thus $\norm{f_\lambda}\rightarrow\norm{L}$. Since $\norm{\overline{f}} \leq \liminf\norm{f_\lambda}=\norm{L}$ (c.f. Brezis~\cite{brezis2011} Proposition 3.5 (iii)) we have $\norm{\overline{f}}=\norm{L}$ and thus $L\in J(\overline{f})$.\\
  Since the $f_\lambda$ are minimisers of the regularisation problem we have for all $g\in\banachspace$ such that $L(g) = \norm{L}^2$
  $$\error\left(\frac{a_0}{\norm{L}^2}L(f_\lambda)\nu,y\right)+\lambda\Omega(f_\lambda) \leq \error(a_0\nu,y)+\lambda\Omega(g)$$
  Since $a_0$ is the minimiser this implies in particular that
  $$\Omega(f_\lambda) \leq \Omega(g) \qquad \forall g\in\banachspace \mbox{ such that } L(g)=\norm{L}^2$$
  and taking the limit inferior again we obtain that $\overline{f}$ is in fact a solution of the interpolation problem
  $$\min\{\Omega(f) \setseparator f\in\banachspace, L(f)=\norm{L}^2\}$$
  Now this means that $\Omega(\overline{f}+f_T) \geq \Omega(\overline{f})$ for all $f_T\in\ker(L)$ and if $\overline{f} = f$ we are clearly done.\\
  If $\overline{f}\neq f$ we know that $f$ and $\overline{f}$ are in the same face as $L\in J(f)$ and $L\in J(\overline{f})$. They thus have the same error $\error$. If $\Omega(f) = \Omega(\overline{f})$  then both are equivalent minimisers and it is clear that both satisfy the tangential bound. If $\Omega(f) > \Omega(\overline{f})$ then $f$ is not admissible and does not need to satisfy the tangential bound.\\
  \ \\
  Finally note that the claim is trivially true for $L=0$ as in that case $\error$ is independent of $f$ and for every $\lambda$ the minimiser $f_\lambda$ has to be zero to satisfy $J(f_\lambda)\cap\{0\}\neq\emptyset$. This means $\Omega$ is minimised at 0.
\end{proof}

\subsection{Duality mappings}
The proof of \cref{thm:duality_map_intersection} crucially relies on the following connection of the duality mapping with subgradients (c.f. \cite{asplund1967,blazek1982}).

\begin{proposition}\label{prop:integral_functional}
  For a normed linear space $V$ with duality mapping $J$ with gauge function $\mu$ define $M:V \rightarrow \R$ by
  \begin{equation}
    M(x) = \int\limits_{0}^{\norm{x}_V} \mu(t)\differential{t}
  \end{equation}
  Then $x^\ast\in\partial M(x)\subset V^\ast$, the subgradient of $M$, if and only if
  $$M(y) \geq M(x) + \dualmap{}{x^\ast}{y-x}$$
  For any $0\neq x \in V$ we have that $\partial M(x) = J(x)$.
\end{proposition}

We now give a proof of \cref{thm:duality_map_intersection} which follows the ideas of the one presented in~\cite{blazek1982} but corrects the mistake from that paper.\\

\begin{proof}[Of \Cref{thm:duality_map_intersection}]
  Using the functional $M$ from \cref{prop:integral_functional} define a functional $F\,\colon V\rightarrow\R$ by
  $$F(x) = M(x-x_0) - L_0(x-x_0)$$
  Since $M$ is continuous, convex with strictly increasing derivative and $L_0$ is linear, $F$ is clearly continuous, convex and coercive. This means that $F$ attains its minimum on the reflexive subspace $W$ in at least one point, $\overline{z}$ say.\\
  Hence for all $y\in W$
  \begin{gather}
    F(y) - F(\overline{z}) \geq 0 \notag\\
    \Leftrightarrow M(y - x_0) \geq M(\overline{z} - x_0) + L_0(y - \overline{z}) \notag\\
    \Leftrightarrow M(y - x_0) - M(\overline{z} - x_0) + L_0(\overline{z} - x_0) \geq L_0(y - x_0)\label{eq:duality_inequality}
  \end{gather}
  By \cref{prop:integral_functional} this means that $L_0\big|_W \in \partial M\big|_W(\overline{z} - x_0) = J_\mu\big|_W(\overline{z} - x_0)$. For simplicity we write $L_0\big|_W = L_W$.\\
  Note that if $x_0\in W$ and $L_W = 0$ we have that $F(x) = M(x-x_0)$ on $W$ so $\overline{z} = x_0$ and we trivially have $J_\mu(x_0-x_0) = \{0\} = \{-L_0 + L_0\} \subset W^\perp + L_0$. So we can without loss of generality assume that not both $x_0\in W$ and $L_W=0$. \\
  \ \\
  In case $x_0\in W$ it is clear that $M$ is minimised at $x_0$. If $L_W \neq 0$ then $L_W$ attains its norm on $W$ in a point $z$ say. Thus it is clear that there exists a minimiser for $F$ of the form $\overline{z} = z + x_0$. More precisely $F$ is minimised where an element of $\partial M$ and $\nabla L_0$ are equal. Since $\partial M(x-x_0)=\mu(\norm{x-x_0})\frac{L_x}{\mu(\norm{x-x_0})}$ for $L_x\in J_\mu(x-x_0)$ we get that the minimiser $\overline{z}=z+x_0$ is such that $\norm{L_W}_{W^\ast}=\mu(\norm{\overline{z}-x_0})$.\\
  \ \\
  If on the other hand $x_0\not\in W$ then we note that $\overline{z}$ being the minimum for $F$ on $W$ implies that $L_z(y) \geq 0$ for all $L_z \in \partial F(\overline{z})$ and all $y\in W$. But this means that
  $$\mu(\norm{\overline{z}-x_0})\frac{L_z(y)}{\mu(\norm{\overline{z}-x_0})} - L_0(y) \geq 0$$
  for every $L_z\in J(\overline{z}-x_0)$. But since $\frac{L_z}{\mu(\norm{\overline{z}-x_0})}$ is of norm 1 this means that
  $$\mu(\norm{\overline{z}-x_0})\cdot\norm{y} \geq \mu(\norm{\overline{z}-x_0})\frac{L_z(y)}{\mu(\norm{\overline{z}-x_0})} \geq L_W(y)$$
  for all $y\in W$. Thus $\norm{L_W}_{W^\ast}=\norm{L_0\big|_W}_{W^\ast}\leq\mu(\norm{\overline{z}-x_0})$.\\
  Now denote by $\overline{W}$ the space generated by $W$ and $x_0$ and note that this space is still reflexive. Extend $L_W$ to $L_{\overline{W}}$ on $\overline{W}$ by setting
  $$L_{\overline{W}}(x_0) = L_0(\overline{z}) - \mu(\norm{\overline{z}-x_0})\cdot\norm{\overline{z}-x_0}$$
  Then
  \begin{align*}
    L_{\overline{W}}(\overline{z}-x_0) & = L_W(\overline{z}) - \left(L_0(\overline{z}) - \mu(\norm{\overline{z}-x_0})\cdot\norm{\overline{z}-x_0}\right)\\
    & = \mu(\norm{\overline{z}-x_0})\cdot\norm{\overline{z}-x_0}
  \end{align*}
  so $\norm{L_{\overline{W}}}_{\overline{W}^\ast} \geq \mu(\norm{\overline{z}-x_0})$.\\
  Further $L_{\overline{W}}(y)=L_W(y)\leq\mu(\norm{\overline{z}-x_0})\cdot\norm{y}$ for all $y\in W$, so $\norm{L_{\overline{W}}} > \mu(\norm{\overline{z}-x_0})$ can only happen if the norm is attained for some point $\lambda y+\nu x_0$ for $y\in W$, $\nu\neq0$. Or equivalently, dividing through by $\nu$, at a point $y+x_0$ for some $y\in W$. But for those points we have
  \begin{align*}
    L_{\overline{W}}(y+x_0) & = L_W(y) + L_0(\overline{z}) - \mu(\norm{\overline{z}-x_0})\cdot\norm{\overline{z}-x_0}\\
                                         & \leq \mu(\norm{\overline{z}-x_0})\cdot\norm{y+\overline{z}} - \mu(\norm{\overline{z}-x_0})\cdot\norm{\overline{z}-x_0}\\
                                         & \leq \mu(\norm{\overline{z}-x_0})\left|\norm{y+\overline{z}} - \norm{\overline{z}-x_0}\right|\\
    & \leq \mu(\norm{\overline{z}-x_0})\cdot\norm{y+x_0}
  \end{align*}
  and thus $\norm{L_{\overline{W}}} = \mu(\norm{\overline{z}-x_0})$ and $L_{\overline{W}}(\overline{z}-x_0) = \norm{L_{\overline{W}}}\cdot\norm{\overline{z}-x_0}$.\\
  \ \\
  Since for $x_0\in W$ we have $\overline{W}=W$ in either case we have obtained a function $L_{\overline{W}}$ such that $L_{\overline{W}}=L_0\big|_W$, $\norm{L_{\overline{W}}}=\mu(\norm{\overline{z}-x_0})$ and $L_{\overline{W}}(\overline{z}-x_0) = \norm{L_{\overline{W}}}\cdot\norm{\overline{z}-x_0}$.\\
  Now extend $L_{\overline{W}}$ by Hahn-Banach to $L_V$ on $V$ such that
  $$\norm{L_V}=\norm{L_{\overline{W}}}=\mu(\norm{\overline{z}-x_0})$$
  and $L_V\big|_{\overline{W}}=L_{\overline{W}}$. Hence $(L_V - L_0)\big|_W = 0$ so $L_V\in W^\perp + L_0$.\\
  It remains to show that $L_V\in J_\mu(\overline{z}-x_0)$ by showing \cref{eq:duality_inequality} holds for $L_V$ and every $y\in V$. Notice first that
  \begin{equation}
    \label{eq:u_star_estimate}
    L_V(y-x_0) \leq \norm{L_V}\cdot\norm{y-x_0} = \norm{L_{\overline{W}}}\cdot\norm{y-x_0} = L_{\overline{W}}\left(\frac{\norm{y-x_0}}{\norm{\overline{z}-x_0}}(\overline{z}-x_0)\right)
  \end{equation}
  But
  \begin{multline}
    \label{eq:v_star_estimate}
    L_{\overline{W}}\left(\frac{\norm{y-x_0}}{\norm{\overline{z}-x_0}}(\overline{z}-x_0)\right) - L_{\overline{W}}(\overline{z}-x_0) = \left(\frac{\norm{y-x_0}}{\norm{\overline{z}-x_0}} - 1\right)\mu(\norm{\overline{z}-x_0})\norm{\overline{z}-x_0}\\
    = \mu(\norm{\overline{z}-x_0})\left(\norm{y-x_0} - \norm{\overline{z}-x_0}\right)
  \end{multline}
  and further
  \begin{equation}
    \label{eq:M_estimate}
    M(y-x_0) - M(\overline{z}-x_0) = \int\limits_{\norm{\overline{z}-x_0}}^{\norm{y-x_0}} \mu(t)\differential{t} \geq \left(\norm{y-x_0} - \norm{\overline{z}-x_0}\right)\mu(\norm{\overline{z}-x_0})
  \end{equation}
  so the left hand side of \cref{eq:M_estimate} is always at least as big as the left hand side of \cref{eq:v_star_estimate}. We can thus add the left hand side of \cref{eq:v_star_estimate} to the right hand side of \cref{eq:duality_inequality} and the left hand side of \cref{eq:M_estimate} to the left hand side of \cref{eq:duality_inequality} while preserving the inequality. \Cref{eq:duality_inequality} is in particular true for $\overline{z}$ and in that case also for $L_{\overline{W}}$ as it agrees with $L_0$ on $\overline{z}$ and $x_0$, i.e.
  $$M(\overline{z} - x_0) - M(\overline{z} - x_0) + L_{\overline{W}}(\overline{z} - x_0) \geq L_{\overline{W}}(\overline{z} - x_0)$$
  Thus by adding the left hand sides of \cref{eq:v_star_estimate} and \cref{eq:M_estimate} as described we obtain
  $$M(y-x_0) - M(\overline{z}-x_0) + L_{\overline{W}}(\overline{z}-x_0) \geq L_{\overline{W}}\left(\frac{\norm{y-x_0}}{\norm{\overline{z}-x_0}}(\overline{z}-x_0)\right)$$
  for all $y\in V$. But since $L_V$ also agrees with $L_{\overline{W}}$ on $\overline{z}$ and $x_0$ this together with \cref{eq:u_star_estimate} implies that
  $$M(y - x_0) - M(\overline{z} - x_0) + L_V(\overline{z} - x_0) \geq L_V(y - x_0)$$
  for all $y\in V$ which is what we wanted to prove. Thus indeed $L_V\in J_\mu(\overline{z}-x_0)$ as claimed. By homogeneity of $J_\mu$ clearly $-L_V$ with $-\overline{z}\in W$ is as in the statement of the theorem.
\end{proof}\end{appendix}

\end{document}